\documentclass[11pt,a4paper]{article}

\usepackage[margin=1in]{geometry}
\usepackage{amsmath,amssymb,amsthm,mathtools}
\usepackage{hyperref}
\hypersetup{colorlinks=true,linkcolor=blue,citecolor=blue,urlcolor=blue}

\theoremstyle{plain}
\newtheorem{theorem}{Theorem}

\newtheorem{lemma}{Lemma}
\newtheorem{corollary}{Corollary}
\newtheorem{definition}{Definition}
\newtheorem{remark}{Remark}
\newtheorem{assumption}{Assumption}
\newtheorem{example}{Example}

\newcommand{\R}{\mathbb{R}}
\newcommand{\N}{\mathbb{N}}
\newcommand{\Pol}{\R[x]}

\newcommand{\norm}[1]{\lVert #1 \rVert}
\newcommand{\Sig}{\Sigma}              

\newcommand{\Ltwo}[1]{L^2(#1)}

\title{Density, Determinacy, Duality\\ and a Regularized Moment-SOS Hierarchy}
\author{Didier Henrion\footnote{LAAS-CNRS, University of  Toulouse, France and Faculty of Electrical Engineering, Czech Technical University in Prague, Czechia.}}
\date{\today}
\begin{document}
	\maketitle
\begin{abstract}
	The standard moment-sum-of-squares (SOS) hierarchy is a powerful method for solving global polynomial optimization problems. However, its convergence relies on Putinar’s Positivstellensatz, which requires the feasible set to satisfy the algebraic Archimedean property. In this paper, we introduce a regularized moment-SOS hierarchy capable of handling problems on unbounded sets or bounded sets violating the Archimedean property. Adopting a functional analysis viewpoint, we rely on the multivariate Carleman condition for measure determinacy rather than algebraic compactness. We prove that finite degree projections of the quadratic module are dense in the cone of positive polynomials with respect to the square norm induced by the measure. Based on these density results, we prove the convergence of a regularized hierarchy without invoking any Positivstellensatz. Furthermore, we propose a penalized formulation of the hierarchy which, combined with Bernstein-Markov inequalities, provides a monotonically non-decreasing sequence of certified lower bounds on the global minimum. The approach is illustrated on several benchmark problems known to be difficult or ill-posed for the standard hierarchy.
\end{abstract}
	
The \emph{moment-SOS hierarchy} is a powerful approach for solving globally difficult non-linear non-convex problems.
It was designed originally for polynomial optimization \cite{L01}, for recent overviews and more references see e.g. \cite{HKL20,N23,T24}.
Its convergence proof relies on Putinar's Positivstellensatz (Psatz) \cite{P93}, which is a distinguished sum-of-squares (SOS) representation of positive polynomials on semialgebraic sets
of the form
\[
K:=\left\{x \in \R^n : g_j(x) \geq 0, \ j=1,\ldots,m\right\}
\]
where $g_j \in \R[x]$ are given real polynomials of the indeterminate $x \in \R^n$. Let $g_0:=1$ for notational convenience. Convergence of the hierarchy is guaranteed under
a specific algebraic compactness assumption on the data which is called the \emph{Archimedean property}: the quadratic module
\[
Q(g):=\left\{  \sum_{j=0}^m s_j g_j : s_j \in \Sigma[x], \ j=0,1,\ldots,m\right\}
\]
must contain the polynomial \[R^2-x^Tx\] for some number $R$,
where $\Sigma[x]$ is the cone of SOS polynomials, see e.g. \cite{PD01,M08}.
This property implies that $K$ is compact, and this restricts the application of the hierarchy in its native form.
Indeed there are compact semialgebraic sets which violate the Archimedean property, one of the simplest being
\[
\{x \in \R^2 : 2x_1-1 \geq 0,\ 2x_2-1 \geq 0,\ 1-x_1x_2 \geq 0\}
\]
see \cite[Ex. 6.3.1]{PD01} and \cite[Ex 6.5]{N08}.
As surveyed e.g. in \cite{N23}, the compactness assumption can be circumvented by using various reformulation strategies (change of coordinates,
homogenization) so that the problem eventually boils down to a modified POP satisfying the Archimedean property.
The compactness assumption was relaxed to a certain extent by introducing the multivariate Carleman condition in the context of polynomial optimization \cite{L06}. The Carleman condition ensures measure \emph{determinacy}, i.e. the property that a measure is uniquely specified by its moments. This allowed to construct simple high degree perturbation to approximate positive polynomials with SOS \cite{LN07}, and this culminated in
an elegant treatment in \cite{L13} of the problem of moments on unbounded semialgebraic sets, a problem dual to the SOS representation of positive polynomials, see e.g. \cite{S17} for a comprehensive account.

In this paper, we introduce a \emph{regularized} moment-SOS hierarchy that can deal with POPs on unbounded sets, or bounded sets defined by polynomials violating the Archimedean property. The convergence proof of our regularized hierarchy does not rely on Putinar's Psatz. In contrast with \cite{L13} which develops a coefficient-space program in a weighted topology of discrete moment sequences, our treatment follows a function-space viewpoint. Similarly to \cite{L13}, our approach relies  on the Carleman condition. We formulate our results as \emph{density} and separation properties in the topology of the Hilbert space $L^2(\mu)$ where $\mu$ is a measure satisfying the Carleman condition.

Our contributions are as follows. After the preliminary Sections \ref{sec:carleman} and \ref{sec:polydensity},
we provide in Section \ref{sec:sosdensity} a direct functional analytic proof that SOS polynomials are uniformly dense in positive polynomials. Without assuming the Archimedean property of $Q(g)$ or the compactness of $K$, we prove in Section \ref{sec:qmdensity} that finite degree projections of the quadratic module $Q(g)$ are uniformly dense in finite degree positive polynomials. In Section \ref{sec:regmomsos} we introduce a regularized version of the moment-SOS hierarchy and prove its convergence with our density results. In particular our convergence proof does not rely on any Psatz. We also describe an alternative penalized version of the regularized moment-SOS hierarchy that provides a monotonically non-decreasing sequence of lower bounds on the global minimum on compact sets. We illustrate our regularized moment-SOS hierarchy on various degenerate benchmark POPs in Section \ref{sec:examples}.

\section{Carleman Determinacy}\label{sec:carleman}

In the following, $\mu$ is a finite positive Borel measure.  The Lebesgue space $L^1(\mu)$ consists of measurable integrable functions $f : \R^n \to \R$.
	This is a Banach space with the norm
	\[
	\|f\|_{L^1(\mu)}\ :=\ \int_{\R^n} |f(x)|\,d\mu(x).
	\]
The Lebesgue space $L^2(\mu)$ consists of measurable square integrable functions $f : \R^n \to \R$.
	This is a Hilbert space with inner product
	$$\langle f,g\rangle_{L^2(\mu)}:=\int f(x)g(x)\,d\mu(x)$$ and norm $$\|f\|_{L^2(\mu)}:=\big(\int f^2(x)\,d\mu(x)\big)^{1/2}.$$
 Since $\mu$ is finite, $L^2(\mu)\subset L^1(\mu)$ and $\|f\|_{L^1(\mu)}\le \mu(\R^n)^{1/2}\|f\|_{L^2(\mu)}$.

\begin{definition}
Given an integer vector $\alpha \in \N^n$, the \emph{moment} of order $\alpha$ of the measure $\mu$ is the real number
\[
y_{\alpha} = \int x^\alpha d\mu(x) := \int x^{\alpha_1}_1 x^{\alpha_2}_2 \cdots x^{\alpha_n}_n d\mu(x_1,x_2,\ldots,x_n).
\]
The moments of $\mu$ are the infinite countable sequence $(y_\alpha)_{\alpha \in \N^n} \subset \R$.
\end{definition} 

\begin{definition}
A measure with given moments is called \emph{determinate} if it is the unique measure with these moments.
\end{definition}
	
\begin{definition}\label{def:marginals}
For $i=1,\ldots,n$, define the even marginal moments
	\[
	y_{i,2k}\ :=\  \int_{\R^n} x_i^{2k}\,d\mu(x_1, \ldots, x_n),\qquad k\in\N.
	\]
	We say that $\mu$ satisfies the multivariate Carleman condition, or \emph{satisfies Carleman} for short, if for
	each $i=1,\dots,n$,
	\[
	\sum_{k=1}^{\infty} y_{i,2k}^{-1/(2k)}\;=\;+\infty.
	\]
\end{definition}

\begin{theorem}
	\label{thm:carleman-multi}
	If $\mu$ satisfies Carleman,
	then it is determinate.
\end{theorem}

For a statement and references to the original proof see e.g. \cite[Prop. 3.5]{L10} and \cite[Sec. 14.4]{S17}.

\begin{lemma}\label{lem:carleman-finite-moments}
If $\mu$ satisfies Carleman, then all its moments
are finite.
\end{lemma}

\begin{proof}
Even marginal moments are finite.
Let $i \in \{1,\ldots,n\}$ be given. If $y_{i,2k_0}=+\infty$ for some $k_0$, then on a finite measure space
$L^p\subset L^q$ for $p>q$, hence $|x_i|\notin L^{2k}$ for every $k\ge k_0$ and thus
$y_{i,2k}=+\infty$ for all $k\ge k_0$. The Carleman series would then have only finitely
many positive terms and converge, contradicting the hypothesis. Hence $y_{i,2k}<\infty$
for all $k\in\N$ and each $i$.
	
Marginal moments are finite.
Let $i \in {1,\ldots,n}$ and $r\ge 0$ be given. Choose $k\in\N$ with $2k\ge r$. Since $|x_i|\in L^{2k}(\mu)$ as above
and $\mu(\R^n)<\infty$, the embeddings on finite measure spaces give
\begin{equation}\label{eq:mombound}
\int_{\R^n} |x_i|^{r}\,d\mu \;=\;\|\,|x_i|\,\|_{L^{r}(\mu)}^{r}
\;\le\; \mu(\R^n)^{1-\frac{r}{2k}}\,\|\,|x_i|\,\|_{L^{2k}(\mu)}^{r}
\;<\;\infty.
\end{equation}
	
Absolute mixed moments are finite. Let $\alpha\in\N^n$ and $m:=|\alpha|=\sum_{i=1}^n\alpha_i$. For $m\ge 1$, apply the weighted arithmetic-geometric inequality to $a_i:=|x_i|^{m}$ with weights
$w_i:=\alpha_i/m$ (so $\sum_i w_i=1$):
\[
|x^\alpha|=\prod_{i=1}^n |x_i|^{\alpha_i}
=\prod_{i=1}^n a_i^{w_i}
\;\le\; \sum_{i=1}^n w_i a_i
\;\le\; \sum_{i=1}^n |x_i|^{m}.
\]
Integrating and using \eqref{eq:mombound} with $r=m$ yields
\[
\int_{\R^n} |x^\alpha|\,d\mu \;\le\; \sum_{i=1}^n \int_{\R^n} |x_i|^{m}\,d\mu \;<\;\infty.
\]
	
Signed moments are finite.
By the triangle inequality,
\[
\biggl|\int_{\R^n} x^\alpha\,d\mu\biggr|
\;\le\; \int_{\R^n} |x^\alpha|\,d\mu \;<\;\infty,
\]
so $\int x^\alpha\,d\mu$ is finite for every $\alpha$.
\end{proof}

\begin{lemma}\label{lem:compact-carleman}
	If $\mu$ has compact support, then it satisfies Carleman.
\end{lemma}

\begin{proof}
	Since the support of $\mu$ is compact, for each $i=1,\ldots,n$ there exists $R_i>0$ such that $|x_i|\le R_i$ $\mu$-a.e. Let $M:=\mu(\R^n)\in(0,\infty)$. Then $y_{i,2k} \;\le\; R_i^{2k}\,M$, $
	k\in\N$, so $y_{i,2k}^{-1/(2k)} \;\ge\; R_i^{-1}\,M^{-1/(2k)}$.
	Since $M^{-1/(2k)}\to 1$ as $k\to\infty$, there is $k_0$ such that $y_{i,2k}^{-1/(2k)} \ge \tfrac12 R_i^{-1}$ for all $k\ge k_0$. Hence $\sum_{k=1}^\infty y_{i,2k}^{-1/(2k)} \;\ge\; \sum_{k=k_0}^\infty \tfrac12 R_i^{-1} \;=\;+\infty.$
	As this holds for each coordinate $i$, the multivariate Carleman condition is satisfied.
\end{proof}

\begin{lemma}\label{lem:gibbs}
	Let $p\in\R[x]$ be a coercive polynomial, i.e., $\lim_{\|x\|\to\infty} p(x)=+\infty$.
	Then the Gibbs measure $\mu(dx)=e^{-p(x)}\,dx$ satisfies Carleman.
\end{lemma}

\begin{proof}
If $p(x)$ is a coercive polynomial of degree $d$, then $d$ is an even integer and the highest-degree part is positive definite. Hence there exist constants $a > 0$ and $b \in \R$ such that for all $x \in \R^n$:
	\[
	p(x) \ge a\norm{x}^d - b.
	\]
	Therefore, the measure $\mu$ is bounded above by a simpler radial measure:
	\[
	d\mu(x) = e^{-p(x)}dx \le e^{b}e^{-a\norm{x}^d}dx.
	\]	
	We must check Carleman's condition for each one-dimensional marginal. Let's bound the moments $y_{i,2k} = \int_{\R^n} x_i^{2k} d\mu(x)$ for a fixed coordinate $i$. Since $|x_i| \le \norm{x}$, we have $x_i^{2k} \le \norm{x}^{2k}$ and
	\[
	y_{i,2k} = \int_{\R^n} x_i^{2k} e^{-p(x)}dx \le e^b \int_{\R^n} \norm{x}^{2k} e^{-a\norm{x}^d}dx.
	\]
	To evaluate the integral, let us use spherical coordinates. Let $r = \norm{x}$ and let $S_{n-1}$ be the surface area of the unit $(n-1)$-sphere.
	\[
	\int_{\R^n} \norm{x}^{2k} e^{-a\norm{x}^d}dx = S_{n-1} \int_0^\infty r^{2k} e^{-ar^d} r^{n-1} dr = S_{n-1} \int_0^\infty r^{2k+n-1} e^{-ar^d} dr.
	\]
	This integral evaluates in terms of the Gamma function to:
	\[
	\frac{S_{n-1}}{d} a^{-\frac{2k+n}{d}} \Gamma\left(\frac{2k+n}{d}\right).
	\]
	Thus, we have an upper bound on the moments where $c$ is a constant independent of $k$:
	\[
	y_{i,2k} \le c a^{-\frac{2k}{d}}  \Gamma\left(\frac{2k+n}{d}\right).
	\]
	Hence
	\[
	y_{i,2k}^{-1/(2k)}
	\;\ge\;
	c^{-1/(2k)}\,a^{1/d}\,
	\Gamma\!\Big(\frac{2k+n}{d}\Big)^{-1/(2k)}.
	\]
	Using Stirling’s formula for the Gamma function and the fact that $d\ge2$, one obtains
	for some constant $c'>0$ and all large $k$,
	\[
	y_{i,2k}^{-1/(2k)} \;\ge\; c'\,k^{-1/d}.
	\]
	Since $1/d\le1/2<1$, the series
	\[
	\sum_{k=1}^\infty y_{i,2k}^{-1/(2k)}
	\;\ge\; c' \sum_{k=1}^\infty \frac{1}{k^{1/d}}
	\]
	diverges. This holds for every coordinate $i$, so $\mu$ satisfies the multivariate
	Carleman condition.
\end{proof}

In \cite[Thm. 3.2.17]{DX14} it is shown that if $\mu$ satisfies the condition
\begin{equation}\label{eq:exp}
\int e^{c|x|} d\mu(x) < \infty
\end{equation}
for some $c>0$ then $\mu$ satisfies Carleman. 

\section{Polynomial Density}\label{sec:polydensity}

\begin{lemma}
	\label{lem:poly-in-L2}
	If $\mu$ satisfies Carleman,
	every polynomial $p\in\R[x]$ belongs to $L^2(\mu)$.
\end{lemma}
\begin{proof}
The squared norm $\|p\|^2_{L^2(\mu)}=\int p(x)^2\,d\mu(x)$ is a finite linear combination of moments and it is therefore finite by Lemma \ref{lem:carleman-finite-moments}.
\end{proof}

\begin{lemma}\label{lem:density}
	A subspace is dense in $L^2(\mu)$ iff its orthogonal complement is trivial.
\end{lemma}

\begin{proof}
	For a linear subspace $V\subset L^2(\mu)$, its orthogonal complement is
	\[
	V^\perp:=\{\,h\in L^2(\mu): \langle h,v\rangle_{L^2(\mu)}=0\ \text{for all } v\in V\,\}.
	\]
	It is a closed subspace of $L^2(\mu)$ and
	\begin{equation}\label{orth}
		\overline{V}\;=\;(V^\perp)^\perp .
	\end{equation}
	To see this, observe first that 
	$V\subset (V^\perp)^\perp$ by definition, hence $\overline V\subset (V^\perp)^\perp$ since $(V^\perp)^\perp$ is closed.
	For the reverse inclusion, let $P:=\overline V$. The orthogonal decomposition theorem says every $h\in L^2(\mu)$ can be written uniquely as
	$h=p+q$ with $p\in P$ and $q\in P^\perp$. If $h\in (V^\perp)^\perp$, then $h$ is orthogonal to $V^\perp$, hence to $P^\perp$ (because $V^\perp\supset P^\perp$). Thus $0=\langle h,q\rangle=\|q\|_2^2$, so $q=0$ and $h\in P=\overline V$. Therefore $(V^\perp)^\perp\subset \overline V$.	
	
	Now if $V^\perp=\{0\}$ then $(V^\perp)^\perp=L^2(\mu)$, and from relation \eqref{orth} it holds $\overline V=L^2(\mu)$, i.e.\ $V$ is dense. Conversely, if $V$ is dense, then $\overline V=(V^\perp)^\perp=L^2(\mu)$, which forces $V^\perp=\{0\}$.
\end{proof}	

\begin{lemma}
	\label{lem:inherit-carleman}
	Let $\mu$ satisfy Carleman. Let $\nu$ be a signed measure which is absolutely continuous w.r.t. $\mu$ with density $f \in L^2(\mu)$, i.e.\ $d\nu = f\,d\mu$. Let
	$\nu = \nu^+ - \nu^-$ be its Jordan decomposition, so that
	$d\nu^\pm = f_\pm\,d\mu$ with $f_+ := \max\{f,0\}$ and $f_- := \max\{-f,0\}$.
	Then each $\nu^\pm$ satisfies Carleman.
\end{lemma}

\begin{proof}
	Given $i \in \{1,\dots,n\}$, $k \in \N$, consider the marginal moments $y_{i,2k}$ as in Definition \ref{def:marginals}. 
	Since $\mu$ satisfies Carleman, Lemma~\ref{lem:carleman-finite-moments}
	implies $y_{i,2k} < \infty$ for all $k$, and the Carleman series
	\[
	\sum_{k=1}^\infty y_{i,2k}^{-1/(2k)} = +\infty.
	\]	
	Let $f_\pm$ be as above, so $d\nu^\pm = f_\pm\,d\mu$ and $f_\pm \in L^2(\mu)$,
	$f_\pm \ge 0$. Define the even marginal moments of $\nu^\pm$ by
	\[
	z_{2k}^{\pm} := \int_{\mathbb R^n} x_i^{2k}\,d\nu^\pm(x)
	= \int_{\mathbb R^n} x_i^{2k} f_\pm(x)\,d\mu(x),
	\qquad k \in \mathbb N.
	\]
	
	\emph{Step 1: finiteness of the moments of $\nu^\pm$.}
	By Cauchy-Schwarz,
	\[
	z_{2k}^{\pm}
	\le \|f_\pm\|_{L^2(\mu)} \,
	\biggl(\int_{\mathbb R^n} x_i^{4k}\,d\mu(x)\biggr)^{1/2},
	\]
	and the right-hand side is finite since all moments of $\mu$ are finite
	(Lemma~\ref{lem:carleman-finite-moments} with $r = 4k$).
	Thus $z_{2k}^{\pm} < \infty$ for all $k$.
	
	\emph{Step 2: a lower bound on the Carleman terms.}
	From the inequality above we obtain, for $k \ge 1$,
	\[
	\bigl(z_{2k}^{\pm}\bigr)^{-1/(2k)}
	\;\ge\;
	\|f_\pm\|_{L^2(\mu)}^{-1/(2k)}
	\biggl(\int_{\mathbb R^n} x_i^{4k}\,d\mu(x)\biggr)^{-1/(4k)}.
	\]
	Let
	\[
	c_k := y_{i,2k}^{-1/(2k)}.
	\]
	Then
	\[
	\bigl(z_{2k}^{\pm}\bigr)^{-1/(2k)}
	\;\ge\;
	\|f_\pm\|_{L^2(\mu)}^{-1/(2k)}\,c_{2k}.
	\]
	Since $\|f_\pm\|_{L^2(\mu)}$ is fixed, we have
	$\|f_\pm\|_{L^2(\mu)}^{-1/(2k)} \to 1$ as $k\to\infty$. Thus there exist
	constants $\gamma \in (0,1)$ and $k_0 \in \mathbb N$ such that
	\begin{equation}\label{eq:y-lower}
		\bigl(z_{2k}^{\pm}\bigr)^{-1/(2k)} \;\ge\; \gamma\,c_{2k}
		\qquad\text{for all }k \ge k_0.
	\end{equation}
	
	\emph{Step 3: divergence of the subsequence $c_{2k}$.}
	We now use only that $c_k = y_{i,2k}^{-1/(2k)}$ comes from moments of a positive
	measure. Define $m_r := \bigl(\int |x_i|^r\,d\mu\bigr)^{1/r}$ for $r>0$.
	By Hölder’s inequality,
	the map $r \mapsto m_r$ is nondecreasing: if $0 < r < s$ then
	$m_r \le m_s$. In particular the sequence
	\[
	m_{2k} = \bigl(y_{i,2k}\bigr)^{1/(2k)}, \qquad k\ge1,
	\]
	is nondecreasing, hence
	\[
	c_k = y_{i,2k}^{-1/(2k)} = m_{2k}^{-1}
	\]
	is nonincreasing in $k$.
	
	We know that the full Carleman series $\sum_{k=1}^\infty c_k$ diverges and that
	$(c_k)$ is nonincreasing and nonnegative. For such a sequence, any even index
	subseries also diverges. Indeed, for each $N$,
	\[
	\sum_{k=1}^N c_{2k}
	\;\ge\; \frac12 \sum_{k=1}^N \bigl(c_{2k} + c_{2k+1}\bigr)
	\;=\; \frac12\Bigl(\sum_{j=2}^{2N+1} c_j\Bigr)
	\;=\; \frac12\Bigl(\sum_{j=1}^{2N+1} c_j - c_1\Bigr).
	\]
	As $N\to\infty$, the sum $\sum_{j=1}^{2N+1} c_j$ tends to $+\infty$, hence
	$\sum_{k=1}^\infty c_{2k}$ also diverges.
	
	\emph{Step 4: Carleman for $\nu^\pm$.}
	Combining \eqref{eq:y-lower} with the divergence of $\sum c_{2k}$ we obtain
	\[
	\sum_{k=1}^\infty \bigl(z_{2k}^{\pm}\bigr)^{-1/(2k)}
	\;\ge\;
	\sum_{k=k_0}^\infty \bigl(z_{2k}^{\pm}\bigr)^{-1/(2k)}
	\;\ge\;
	\gamma \sum_{k=k_0}^\infty c_{2k}
	\;=\; +\infty.
	\]
	Thus the Carleman series for the $i$-th marginal of $\nu^\pm$ diverges.
\end{proof}

\begin{theorem}\label{thm:poly-dense}
If $\mu$ satisfies Carleman, then polynomials are dense in $L^2(\mu)$. Equivalently, for every $f \in L^2(\mu)$ and every $\varepsilon>0$, there exists $p \in \R[x]$ such that $\|f-p\|_{L^2(\mu)}<\varepsilon$.
\end{theorem}

\begin{proof}
	Assume $f\in L^2(\mu)$ is orthogonal to all polynomials:
	\[
	\int_{\R^n} f(x)\,p(x)\,d\mu(x)=0\qquad\forall\,p\in\R[x].
	\]

	Define a finite signed measure $\nu$ by $d\nu:=f\,d\mu$. Finiteness holds because $\mu$ is finite and
	$\int |f|\,d\mu\le \mu(\R^n)^{1/2}\|f\|_{2}<\infty$.
	
	For any multi-index $\alpha\in\N^n$, Cauchy-Schwarz and Lemma~\ref{lem:poly-in-L2} give
	\[
	\int_{\R^n} |x^\alpha|\,d|\nu|(x)
	= \int_{\R^n} |x^\alpha|\,|f(x)|\,d\mu(x)
	\le \Big(\int_{\R^n} x^{2\alpha}\,d\mu\Big)^{1/2}\ \|f\|_{L^2(\mu)} \ <\ \infty.
	\]
	Hence the $\alpha$-moment of $\nu$ is finite, and by orthogonality
	\[
	\int_{\R^n} x^\alpha\,d\nu(x) \;=\; \int_{\R^n} f(x)\,x^\alpha\,d\mu(x)\;=\;0
	\qquad \forall\,\alpha\in\N^n.
	\]
Write the Jordan decomposition $\nu=\nu^+-\nu^-$ with $\nu^\pm$ finite positive measures.
Since $|\nu|$ has all moments finite and $\nu^\pm \le |\nu|$ as measures, each $\nu^\pm$
also has all (absolute) moments finite. Moreover,
\[
\int_{\R^n} x^\alpha\,d\nu(x)
= \int_{\R^n} x^\alpha\,d\nu^+(x)
- \int_{\R^n} x^\alpha\,d\nu^-(x)
= 0
\qquad\forall\alpha\in\N^n,
\]
so
\[
\int x^\alpha\,d\nu^+(x)
= \int x^\alpha\,d\nu^-(x)
\qquad\forall\alpha\in\N^n.
\]
	Since $\mu$ satisfies Carleman, by Lemma \ref{lem:inherit-carleman}, $\nu^\pm$ satisfy Carleman as well. By Theorem \ref{thm:carleman-multi} they are determinate. Since they share all moments they be must be equal:  $\nu^+=\nu^-$ and therefore $\nu=0$.
	
	Since $\nu=f\,\mu=0$ as a signed measure, we have $\int_E f\,d\mu=0$ for all measurable $E$.
	Taking $E=\{f>0\}$ and $E=\{f<0\}$ yields $\mu(\{f\ne0\})=0$, i.e., $f=0$ in $L^2(\mu)$.
	
	Thus the orthogonal complement of the polynomials is $\{0\}$, so polynomials are dense in $L^2(\mu)$.
\end{proof}

It is recalled in \cite[Sec. 3.2.3, p. 69]{DX14} or \cite[Sec. 14.1]{S17} that in the univariate case, if $\mu$ is determinate then polynomials are dense in $L^2(\mu)$, but this is not true in the multivariate case. If $\mu$ satisfies \eqref{eq:exp}, it is shown in \cite[Thm. 3.2.18]{DX14} that polynomials are dense in $L^2(\mu)$.
In \cite[Def. 14.1]{S17} a distinction is made between determinate, strictly determinate and strongly determinate measures, depending on $L ^2$ density of polynomials. In \cite[Thm. 14.19]{S17} it is shown that if a measure satisfies Carleman, then it is strongly determinate. From \cite[Thm. 14.2]{S17} a strongly determinate measure is strictly determinate, which implies $L^2(\mu)$ density of polynomials. So our Theorem \ref{thm:poly-dense} can be alternatively proven with \cite[Thm. 14.2]{S17} and \cite[Thm. 14.19]{S17}. 

	\section{SOS Density}\label{sec:sosdensity}
	
	Let $$P:=\left\{p\in\Pol:\ p(x)\ge0,\ \forall x\in\R^n\right\}$$ denote the cone of \emph{positive polynomials}, and let $$\Sig:=\left\{\sum_k q_k^2:\ q_k\in\Pol\right\} \subset P$$ denote the cone of \emph{sums of squares (SOS)} of polynomials.
	
	We assume in the remainder of the paper that $\mu$ satisfies Carleman, so that the density Theorem~\ref{thm:poly-dense} applies. 
	The following instrumental result was stated and proved in \cite[Theorem 2.2]{L13}. 		
	\begin{lemma}
		\label{lem:h-nonneg}
		If $h\in \Ltwo{\mu}$ satisfies $\int q(x)^2\,h(x)\,d\mu(x)\ge0$ for all $q\in\Pol$, then $h\ge0$ $\mu$-a.e.
	\end{lemma}

	\begin{theorem}
		\label{thm:sos-density}
		If $\mu$ satisfies Carleman, then the cone $\Sigma$ is dense in 
		$P$ w.r.t. $L^2(\mu)$. Equivalently, for every $p\in P$ and every $\varepsilon>0$, there exists $s\in\Sigma$ such that
		$\|p-s\|_{L^2(\mu)}<\varepsilon$.
	\end{theorem}

\begin{proof}
	Let $C$ denote the  $L^2(\mu)$ closure of the SOS cone $\Sigma$, a nonempty, closed, convex cone in the Hilbert space $L^2(\mu)$.
    Fix $p\in P$ and let us prove that $p\in C$.
Assume by contradiction that $p\notin C$.
By the Separating Hyperplane Theorem in Hilbert space -- see e.g. \cite[Thm. 3.38]{BC11} --
there exist $h\in L^2(\mu)\setminus\{0\}$ such that
\begin{equation}\label{eq:sep}
		\langle h, p\rangle_{L^2(\mu)}\ <\ 0\ \leq\ \langle h, s\rangle_{L^2(\mu)}\, \quad\forall\, s\in C.
\end{equation}
Since $\Sigma$ contains $q^2$ for every polynomial $q$, the first inequality in \eqref{eq:sep} gives
$\int q^2\, h\,d\mu\ge 0$ for all $q\in\R[x]$.
By Lemma~\ref{lem:h-nonneg}
this implies $h\ge 0$ $\mu$-a.e.
Because $p\in P$ is nonnegative everywhere and $h\ge 0$ a.e., we have
$\langle   h,p\rangle_{L^2(\mu)}=\int p\,  h\,d\mu\ge 0$, contradicting \eqref{eq:sep}.
Therefore $p\in C$.
\end{proof}

Let $\R[x]_d$ denote the vector space of polynomials of degree up to $d$. 
Let $b=(b_\alpha)_{|\alpha|\le d}$ denote a basis of $\R[x]_d$ orthonormal w.r.t. $\mu$. In this basis any polynomial $p = \sum_{|\alpha|\leq d} p_\alpha b_\alpha$ can be expressed with the coefficient vector $(p_\alpha)_{|\alpha| \leq d}$, and by orthonormality and bilinearity,
\begin{equation}\label{eq:normcoef}
\|p\|_{L^2(\mu)}^2
=\langle \sum_\alpha p_\alpha b_\alpha,\ \sum_\beta p_\beta b_\beta \rangle_{L^2(\mu)}
=\sum_{\alpha,\beta} p_\alpha p_\beta\,\langle b_\alpha,b_\beta\rangle_{L^2(\mu)}
=\sum_\alpha p_\alpha^2
=:\|p\|_2^2.
\end{equation}
In other words, the $L^2(\mu)$ norm of a polynomial is the Euclidean norm of its coefficient vector in an orthonormal basis of $L^2(\mu)$. Note that we use the notation $p$ indifferently for the polynomial function or the vector of its coefficients in basis $b$.

The $L^2(\mu)$ orthogonal \emph{projection} onto $\R[x]_d$ is the linear operator
\[
\Pi_d : L^2(\mu) \to \R[x]_d \subset L^2(\mu)
\]
defined as follows: for any $f \in L^2(\mu)$, the polynomial $\Pi_d f \in \R[x]_d$ is the unique minimizer
\[ 
	\Pi_d f \;=\; \arg\min_{p \in \R[x]_d} \|f - p\|_{L^2(\mu)}.
\]
Equivalently, the projection is characterized by the orthogonality conditions
\[
	\int \bigl(f(x) - \Pi_d f(x)\bigr)\,q(x)\,d\mu(x) \;=\; 0
	\quad\text{for all } q \in \R[x]_d
\]
or the orthogonal basis expression
\[
	\Pi_d f(x) \;=\; \sum_{|\alpha|\le d} \langle f,b_\alpha\rangle_{L^2(\mu)}\,b_\alpha(x).
\]

\begin{theorem}\label{thm:finite-degree-density}
If $\mu$ satisfies Carleman, then the cone $\Pi_{2d}(\Sigma)$ is dense in $P \cap \R[x]_{2d}$ w.r.t $L^2(\mu)$.
Equivalently, for every $p\in P \cap \R[x]_{2d}$ and every $\varepsilon>0$, there exists $s\in\Sigma$ such that
$\|p-\Pi_{2d}s\|_{L^2(\mu)}<\varepsilon$.
\end{theorem}

\begin{proof}
	By Theorem~\ref{thm:sos-density}, $\Sigma$ is dense in $P$ w.r.t. $L^2(\mu)$. 	Given $p\in P \cap \R[x]_{2d}$ and $\varepsilon>0$, pick $s\in\Sigma$ with $\|p-s\|_{L^2(\mu)}<\varepsilon$. 
	Since $p\in\R[x]_{2d}$ and $\Pi_{2d}$ is the $L^2(\mu)$ orthogonal projection,
	\[
	\|\,p-\Pi_{2d}s\,\|_{L^2(\mu)}\ \le\ \|\,p-s\,\|_{L^2(\mu)}\ <\ \varepsilon,
	\]
	which proves the density of $\Pi_{2d}(\Sigma)$ in $P \cap \R[x]_{2d}.$
\end{proof}

\begin{remark}
An $L^2(\mu)$ orthogonal projection
does not preserve positivity (or SOS) in general. For a simple example, take for 
$\mu$ the {uniform probability measure} on $[-1,1]$
and the SOS polynomial $p(x)=x^4$. 
Using the Legendre basis
$l_0(x)=1$, $l_1(x)=x$,$l_2(x)=\tfrac12(3x^2-1)$,
with inner products $\langle l_n,l_m\rangle=\int_{-1}^1 l_nl_m\,dx=
\frac{2}{2n+1}\,\delta_{nm}$, we compute the projection coefficients
\[
p_n=\frac{\langle p,l_n\rangle}{\langle l_n,l_n\rangle},\qquad n=0,1,2.
\]
Since $p$ is even, $\langle p,l_1\rangle=p_1=0$. Then
\[
\langle p,l_0\rangle=\int_{-1}^1 x^4\,dx=\frac{2}{5},\qquad
\langle p,l_2\rangle=\frac12\!\left(3\!\int_{-1}^1 x^6dx-\!\int_{-1}^1 x^4dx\right)
=\frac12\!\left(3\cdot\frac{2}{7}-\frac{2}{5}\right)=\frac{8}{35}.
\]
Hence, using $\langle l_0,l_0\rangle=2$ and $\langle l_2,l_2\rangle=\frac{2}{5}$,
\[
p_0=\frac{\langle p,l_0\rangle}{\langle l_0,l_0\rangle}=\frac{1}{5},\qquad
p_2=\frac{\langle p,l_2\rangle}{\langle l_2,l_2\rangle}=\frac{8/35}{2/5}=\frac{4}{7}.
\]
Therefore
\[
\Pi_2  p
\;=\; p_0 l_0 + p_2 l_2
\;=\; \frac{1}{5} + \frac{4}{7}\cdot\frac{3x^2-1}{2}
\;=\; \frac{6}{7}\,x^2 - \frac{3}{35}
\]
which is negative at $x=0$.

\end{remark}

\begin{remark}
	The cone $\Sigma\cap\R[x]_{2d}$ is a closed cone in a finite-dimensional space. 
	Since there exist positive polynomials whose neighborhoods are not SOS, $\Sigma\cap\R[x]_{2d}$ cannot be dense in $P\cap\R[x]_{2d}$. 
	Theorem~\ref{thm:finite-degree-density} states that if we allow an SOS of possibly higher degree, then its degree $2d$ truncation approximates any polynomial in $P\cap\R[x]_{2d}$ arbitrarily well.
\end{remark}

\begin{example}\label{ex:motzkin}
	The Motzkin polynomial $$p_M(x):=1-3x^2_1x^2_2+x^4_1x^2_2+x^2_1x^4_2$$ is a well-known example of positive but not SOS polynomial, i.e. $p_M \in P \setminus \Sigma$, see e.g. \cite[Section 2.4.2]{N23}.
	Consider $$s(x) :=  \tfrac12\,(x^2_1x_2+x_1x^2_2)^2 + \tfrac12\,(x^2_1x_2-x_1x^2_2)^2
	+(1-\tfrac32\,x^2_1x^2_2)^2.$$
	It holds $\tfrac12\,(x^2_1x_2+x_1x^2_2)^2+\tfrac12\,(x^2_1x_2-x_1x^2_2)^2
	=x^4_1x^2_2+x^2_1x^4_2$ and $(1-\tfrac32\,x^2_1x^2_2)^2 = 1-3x^2_1x^2_2+\tfrac{9}{4}x^4_1x^4_2$ and hence $s(x) = p_M(x)+\tfrac{9}{4}\,x^4_1x^4_2$. Therefore $s\in\Sigma\cap\R[x]_{8}$ matches $p_M \in P \cap \R[x]_6$ in all degrees $\le6$, i.e. $$\|p_M-\Pi_6 s\|_{L^2(\mu)}=0,$$ where $\mu$ is any measure on $\R^2$ satisfying Carleman. 
\end{example}

\section{Quadratic Module Density}\label{sec:qmdensity}

Let $g=(g_1,\dots,g_m)\subset\R[x]$ and let $g_0:=1$. Define the closed basic semialgebraic set
\begin{equation}\label{eq:set}
K:=\left\{x\in\R^n:\ g_j(x)\ge 0\ \text{for }j=1,\dots,m\right\},
\end{equation}
the quadratic module
\[ 
Q(g):=\left\{\sum_{j=0}^m s_j\,g_j\ :\ s_j\in\Sigma\right\}\ \subset\ \R[x],
\]
and the cone of polynomials positive on $K$
\[
P(K) := \left\{\,p\in\R[x]:\ p(x)\ge 0,\ \forall x\in K\,\right\}.
\]

\begin{theorem}\label{thm:denseqm}
	If $\mu$ satisfies Carleman, then the cones $Q(g)$ and $P(K)$ have the same closures w.r.t. $L^2(\mu)$.
	Equivalently, for every $p\in P(K)$ and every $\varepsilon>0$ there exists $s\in Q(g)$ such that
	$\|p-s\|_{L^2(\mu)}<\varepsilon$.
\end{theorem}

\begin{proof}
	All the closures in this proof are w.r.t. $L^2(\mu)$.
	The inclusion of the  closure of $Q(g)$ into the  closure of $P(K)$ is immediate since $Q(g)\subset P(K)$ by construction.
	
	For the reverse inclusion, choose $p$ in the closure of $P(K)$ and suppose by contradiction that $p$ does not belong to $C$, defined as the  closure of $Q(g)$.
	By the Separating Hyperplane Theorem in  $L^2(\mu)$ -- see e.g. \cite[Thm. 3.38]{BC11} -- there exists $h\in L^2(\mu)\setminus\{0\}$ such that
	\begin{equation}\label{eq:sep-Q}
		\langle h, p\rangle_{L^2(\mu)}\ <\ 0\ \leq\ \langle h, s\rangle_{L^2(\mu)}\, \quad\forall\, s\in C.
	\end{equation}
	Since $\Sigma\subset Q(g)\subset C$, taking $s=q^2$ yields $\int q(x)^2\,h(x)\,d\mu(x)\ \ge\ 0$ for all $q\in\R[x].$
	By Lemma~\ref{lem:h-nonneg}, $h\ge 0$ $\mu$-a.e.
	
	Next, for each $j$ we also have $q^2 g_j\in Q(g)\subset C$, hence
	$\int q(x)^2\,g_j(x)\,h(x)\,d\mu(x)\ \ge\ 0$ for all $q\in\R[x].$
	By Lemma~\ref{lem:h-nonneg}  we get $g_j h\ge 0$ $\mu$-a.e. for every $j$.
	In particular, on $\{x:\ g_j(x)<0\}$ one must have $h(x)=0$.
	Therefore $h=0$ $\mu$-a.e.\ on $\{x:\ \min_j g_j(x)<0\}=\R^n\setminus K$, i.e., $\mathrm{spt}\,h\subseteq K$.
	
Combining $h\ge 0$ a.e.\ with $\mathrm{spt}\,h\subseteq K$ we now use that $p$ lies in the closure of $P(K)$:
consider a sequence $p_k\in P(K)$ with $p_k\to p$ in $L^2(\mu)$. Then
$\langle h,p_k\rangle_{L^2(\mu)} \ge\ 0$
for every $k$, and by Cauchy-Schwarz, $\langle h,p_k\rangle\to\langle h,p\rangle$ as $k\to\infty$.
Hence $\langle h,p\rangle\ge0$, contradicting \eqref{eq:sep-Q}. Therefore $p$ belongs to the closure of $Q(g)$.
This proves that the closures of $P(K)$ and $Q(g)$ coincide.
\end{proof}

\begin{remark}
No assumption on compactness of $K$ or Archimedean property of $Q(g)$ is required.  
\end{remark}

\begin{theorem}\label{thm:coeff-L2mu-dense}
	Let $\mu$ satisfy Carleman.
	Given $d\in\N$, the cone $\Pi_d(Q(g))$
	is  dense in $P(K)\cap \R[x]_d$ w.r.t. $L^2(\mu)$.
	Equivalently, for every $p\in P(K)\cap \R[x]_d$ and every $\varepsilon>0$ there exists $s\in Q(g)$ such that
	\[
	\big\|\,p-\Pi_d s\,\big\|_{L^2(\mu)}\ <\ \varepsilon.
	\]
\end{theorem}

\begin{proof}
By Theorem \ref{thm:denseqm}, the $L^2(\mu)$ closures of $Q(g)$ and $P(K)$ coincide. Fix $p\in P(K)\cap\R[x]_d$ and $\varepsilon>0$. Since $p$ belongs to the closure of $Q(g)$, there exists $s\in Q(g)$ satisfying $\|p-s\|_{L^2(\mu)}<\varepsilon$.
	Because $p\in\R[x]_d$ and $\Pi_d$ is an orthogonal projection (hence a contraction),
	\[
	\|\,p-\Pi_d s\,\|_{L^2(\mu)}
	\;=\;\|\,\Pi_d(p-s)\,\|_{L^2(\mu)}
	\;\le\;\|\,p-s\,\|_{L^2(\mu)}
	\;<\;\varepsilon.
	\]
	Since $\Pi_d s\in \Pi_d(Q(g))$, this shows that $\Pi_d(Q(g))$ is dense in $P(K)\cap\R[x]_d$ with respect to $L^2(\mu)$.
\end{proof}

\begin{example}\label{ex:origin}
	Let $g(x)=-x^2$ and $p(x)=x$. The quadratic module $Q(g)$ is Archimedean since it contains $1-x^2$.	For any $\varepsilon>0$, define the quadratic
	\[
	q_\varepsilon(x)\ :=\ \varepsilon\;+\;\frac{x}{2\varepsilon}\;-\;\frac{x^2}{8\,\varepsilon^3},
	\qquad
	s_\varepsilon(x)\ := q_\varepsilon(x)^2  \in\ Q(g).
	\]
	A direct expansion yields
	\[
	s_\varepsilon(x)
	= \varepsilon^2 + x - \frac{1}{8\varepsilon^4}\,x^3 + \frac{1}{64\varepsilon^6}\,x^4,
	\]
	so its degree-$2$ projection is
	\[
	\Pi_2\big(s_\varepsilon\big)\ =\ \varepsilon^2 + x.
	\]
	With $\mu$ the Gaussian measure, we obtain the exact error
	\[
	\|x-\Pi_2(s_\varepsilon)\|_{L^2(\mu)}^2 
	=\varepsilon^4\ \xrightarrow[\varepsilon\to 0]{}\ 0.
	\]
	Moreover, inspection reveals that an exact match $\Pi_2(s)=x$ is impossible in $Q(g)$.
\end{example}

\begin{example}\label{ex:stengle}
	Let $g(x):=(1-x^2)^3$ and $p(x):=1-x^2$. The quadratic module $Q(g)$ is Archimedean since $4-x^2 = (2x^3-3x)^2 + 2x^2 + 4(1-x^2)^3  \in Q(g)$. However, it can be checked that $p \in P(K) \setminus Q(g)$, see \cite{S96} and the developments in \cite[Section 2.4]{BS24}. Yet, with the choice $$s(x):=\tfrac23+\tfrac13 g(x) = 1-x^2 + x^4 -\tfrac13 x^6$$ it holds $$\Pi_2(s)=p.$$
\end{example}

\begin{example}\label{ex:prestel}
	Let $g(x):=(x_1-\tfrac12,\:x_2-\tfrac12,\:1-x_1x_2)$ and $p(x):=\tfrac{17}{4}-x^2_1-x^2_2$. The set $K$ is compact, but it can be checked that $Q(g)$ is not Archimedean -- see \cite[Ex. 6.3.1]{PD01} and \cite[Ex 6.5]{N08} -- and hence $p \in P(K) \setminus Q(g)$. Yet, with the choice $$s(x):=\tfrac{17}{4}+2x^2_1\, g_1(x) + 2x^2_2\, g_2(x) = \tfrac{17}{4}-x^2_1-x^2_2 + 2x^3_1 + 2x^3_2 \in Q(g)$$
	it holds $$\Pi_2(s) = p.$$
\end{example}

\begin{example}
	Let $g_1(x):=x^3$, $g_2(x):=1-x^2$ and $p(x):=x$. Therefore $K=[0,1]$ and the quadratic module $Q(g)$ is Archimedean since it trivially contains $1-x^2=g_2$. However, $p \in P(K) \setminus Q(g)$.
To prove this by contradiction, assume that there exist SOS polynomials $s_0,s_1,s_2$ such that $x = s_0(x) + s_1(x)\,x^3 + s_2(x)\,(1-x^2).$ At $x=0$ it yields
$0 = s_0(0) + s_2(0)$, and since $s_0$ and $s_2$ are SOS, this implies $s_0(0)=s_2(0)=0$. It follows that $s_0(x) = x^2 a(x)$ and $s_2(x) = x^2 b(x)$ with $a,b \in \Sigma[x]$. The SOS decomposition becomes  $x = x^2 a(x) + x^3 s_1(x) + (1-x^2)\,x^2 b(x)
= x^2(a(x) + x s_1(x) + (1-x^2)b(x)).$
Thus the right-hand side is divisible by $x^2$, whereas the left-hand side $x$ is not divisible by $x^2$,
a contradiction. Therefore no SOS representation exists, and $x\notin Q(g)$.

However, for any $\varepsilon\ge \tfrac12$, letting $\delta:=\sqrt{4\varepsilon^2-1}$, we have the simple SOS decomposition
\[
x+\varepsilon =  
\underbrace{\frac{x^2(x+\delta)^2}{4\varepsilon}}_{\displaystyle s_0(x)\in\Sigma[x]}
+
\underbrace{\frac{2(2\varepsilon-\delta)}{4\varepsilon}}_{\displaystyle s_1(x)\in\Sigma[x]}\,x^3
+
\underbrace{\frac{(x+2\varepsilon)^2}{4\varepsilon}}_{\displaystyle s_2(x)\in\Sigma[x]}\,(1-x^2).\]
Putinar's Theorem   guarantees that a SOS decomposition exists for every value of $\varepsilon>0$, but it becomes more complicated (higher degree SOS and higher algebraic degree in $\varepsilon$) when $\varepsilon\to0$.	

	Although $p\notin Q(g)$, we can approximate $p$ arbitrarily well by  first degree truncations of elements of $Q(g)$.
	For any $\varepsilon>0$, let
	\[
	s_\varepsilon(x):=(\varepsilon+\frac{x}{2\varepsilon})^2 = \varepsilon^2\;+\;x\;+\;\frac{1}{4\varepsilon^2}\,x^2 \ \in\ \Sigma\ \subset\ Q(g)
	\]
	and hence
	\[
	\Pi_1(s_\varepsilon)=\varepsilon^2+x.
	\]
	Therefore, for any Carleman measure $\mu$ with $1\in L^2(\mu)$,
	\[
	\big\|p-\Pi_1(s_\varepsilon)\big\|_{L^2(\mu)}
	=\|\varepsilon^2\|_{L^2(\mu)}
	=\varepsilon^2\,\|1\|_{L^2(\mu)}
	\;\xrightarrow[\varepsilon\downarrow0]{}\;0.
	\]
	Note however that an exact identity $\Pi_1(s)=x$ with $s\in Q(g)$ is impossible.
	Indeed, write $s=s_0+x^3s_1+(1-x^2)s_2$ with $s_i\in\Sigma$.
	Only $s_0$ and $(1-x^2)s_2$ contribute to degrees zero and one.
	Let $a_0$ (resp.\ $b_0$) be the constant coefficient of the polynomial whose square is $s_0$ (resp.\ $s_2$).
	Then the constant term of $\Pi_1(s)$ equals $a_0^2+b_0^2\ge0$.
	If $\Pi_1(s)$ had zero constant term (as $x$ does), we would need $a_0=b_0=0$, and for any polynomial $r$, $r(0)=0$ implies $(r^2)'(0)=2\,r(0)\,r'(0)=0,$
	so both linear contributions from $s_0$ and $(1-x^2)s_2$ would vanish, forcing the linear coefficient of $\Pi_1(s)$ to be $0$, a contradiction.
	Thus $\Pi_1(s)=x$ cannot hold, although $\Pi_1(s_\varepsilon)\to x$ in $L^2(\mu)$.
\end{example}

\section{A regularized moment-SOS hierarchy}\label{sec:regmomsos}

Given a polynomial $p \in \R[x]$ and a basic semialgebraic set $K$ as in \eqref{eq:set}, we are consider the polynomial optimization problem (POP)
\[
p^*:=\inf_{x\in K} p(x).
\]

\begin{assumption}[Full-dimensional reference measure]\label{ass:refmeas-full}
The set $K$ has non-empty interior and the probability measure $\mu$ satisfies Carleman and it is absolutely continuous w.r.t. the Lebesgue measure with a continuous density $w$ such that there exists an open neighborhood $U\supset K$ and a constant $w_m$ such that $w(x) \geq w_m > 0$ for all $x\in U$.
\end{assumption}

Assumption \ref{ass:refmeas-full} implies that (i) no nonzero polynomial vanishes $\mu$-a.e.\ on $K$; 
(ii) every nonempty relatively open subset $V\subset K$ satisfies $\mu(V)\ge w_m\,|V|>0$; 
(iii) $Q(g)$ is dense in $P(K)$ in $L^2(\mu)$, see Theorem~\ref{thm:denseqm}.

\subsection{Dual SOS problem}

Let $d_j:=\lceil\deg g_j/2\rceil$ for $j=0,1,\ldots,m$.
Given a relaxation degree $d$ and a parameter  $\varepsilon_d>0$, consider the SOS optimization problem
\begin{equation}\label{eq:sos}
	\begin{aligned}
		v^*_d(\varepsilon_d)\ :=\ \sup\ \ & v \\[2pt]
		\text{\rm s.t.}\quad &
		p-v = r  + \sum_{j=0}^m s_j g_j,\\
		& \|r \|_{L^2(\mu)}\ \le\  \varepsilon_d
	\end{aligned}
\end{equation}
where the maximization is w.r.t. scalar $v$, residual polynomial $r \in \R[x]_{2d}$ and SOS multipliers $s_j\in\Sigma[x]_{2(d-d_j)}$.
Note that $\|r\|_{L^2(\mu)} = \|r \|_2$ can be readily implemented by a second-order cone constraint on the coefficient vector of $r $ in an orthonormal basis w.r.t. $L^2(\mu)$, recall identity \eqref{eq:normcoef}.

In SOS problem \eqref{eq:sos}, polynomial $p-v -r $ belongs to the truncated quadratic module
\[
Q(g)_d:=\left\{ \sum_{j=0}^m s_j g_j:\ s_j\in\Sigma[x],\ \deg(s_j g_j)\le 2d\right\} \subset P(K).
\]

\subsection{Primal pseudo-moment problem}

Introduce the Riesz functional $\ell_y:\R[x]_{2d}\to\R$, $\ell_y(q):=\sum_{|\alpha|\le2d}q_\alpha\,y_\alpha$. 
In the primal problem \eqref{eq:sos} write each SOS in Gram form $s_j(x) = b^T_{d-d_j}(x) Q_j b_{d-d_j}(x)$
with $Q_j \succeq 0$.
The Lagrangian  with dual variable $y$ acting on the identity
$p-v_d=r +\sum_j s_j g_j$ reads
\[
L(v_d,r,(s_j),y):=v_d+\ell_y (p-v_d-r -\sum_{j=0}^m s_j g_j ),
\]
together with conic indicators for $Q_j\succeq0$ and  $\|r \|_2\le  \varepsilon_d$.
Maximizing $L$ over $v_d$ forces $\ell_y(1)=1$.
Maximizing over the Gram matrices gives the standard moment and localizing matrix constraints
$M_{d-d_j}(g_j\,y)\succeq0$.
Maximizing over $r_d$ against the Euclidean ball $\|r \|_2\le\varepsilon_d$ produces the support function
$\varepsilon_d   \|y\|_2$.
Thus the primal problem reads:
\begin{equation}\label{eq:mom}
\begin{array}{lll}
			u^*_d(\varepsilon_d) := & \inf &   \ell_y(p)\;+\; \varepsilon_d  \|y\|_2 \\[2pt]
	& \text{\rm s.t.}  & \ell_y(1)=1,\\
	&&   M_{d-d_j}(g_j\,y)\succeq0,\quad j=0,1,\dots,m
\end{array}
\end{equation}
where the minimization is w.r.t. the vector of pseudo-moments $(y_{\alpha})_{|\alpha|\leq 2d}$.
Alternatively, it can be formulated with a second-order cone constraint
\[
	\begin{array}{ll}
 \inf &   \ell_y(p)\;+\;  \varepsilon_d  t  \\[2pt]
		\text{\rm s.t.}  & \ell_y(1)=1,\\
		&   M_{d-d_j}(g_j\,y)\succeq0,\quad j=0,1,\dots,m \\ 
		& \|y\|_2 \leq t.	
	\end{array}
\]
Weak duality ($u^*_d\geq v^*_d$) follows readily between dual SOS problem \eqref{eq:sos} and primal moment problem \eqref{eq:mom}. Strong duality ($u^*_d = v^*_d$) follows from Slater conditions:
\begin{itemize}
	\item there exists a pseudo-moment vector $y$ such that
	$\ell_y(1)=1$, $M_{d-d_j}(g_j\,y)\succ0$, $j=0,1,\ldots,m$, $\|y\|_2 < t$.
	A concrete sufficient condition is the existence of a strict feasible point in $K$.
	Then, for any small ball $B$ around an interior point of $K$, the moments of the normalized Lebesgue measure on $B$ will be strictly feasible.
	\item there exist $(v,r,(s_j))$ such that
	$\|r\|_{L^2(\mu)}< \varepsilon_d$, $s_j\in\Sigma[x]_{2(d-d_j)}\ \text{with Gram matrices }Q_j\succ0$.
	A practical way to enforce strict primal feasibility (for large $d$) is:
	pick any $\delta>0$, use the $L^2(\mu)$-density of $Q(g)$ in $P(K)$ to find
	$q\in Q(g)$ with $\|p-(p^*-\delta)-q\|_{L^2(\mu)}< \varepsilon_d $, let $r:=p-(p^*-\delta)-q$, and then perturb the SOS Gram
	matrices by a tiny $\epsilon I$ (absorbing the change into $r$) so that
	$Q_j\succ0$ while keeping $\|r\|_{L^2(\mu)}< \varepsilon_d$.
\end{itemize}

\subsection{Convergence}

\begin{theorem}\label{thm:L2-conv}
	Let Assumption~\ref{ass:refmeas-full} hold. 
	If $\varepsilon_d\downarrow0$, then $v_d^\star(\varepsilon_d)\to p^\star$ as $d\to\infty$.
\end{theorem}

\begin{proof}
	\emph{Upper limit.}
	Fix $v>p^\star$. Then $r:=p-v$ is strictly negative at some point of $K$, hence by continuity there exist 
	$\gamma>0$ and a relatively open nonempty $V\subset K$ such that $r(x)\le-\gamma$ on $V$.
	For any $q\in Q(g)$ we have $q\ge0$ on $K$, so $r-q\le-\gamma$ on $V$. Using $w\ge w_m$ on $K$,
	\[
	\|r-q\|_{L^2(\mu)}^2
	=\int_K |r-q|^2\,w\,dx
	\;\ge\;\int_V \gamma^2\,w\,dx
	\;\ge\; \gamma^2 w_m\,|V|
	\;=:\;\delta(v)^2\;>\;0.
	\]
	In particular $\operatorname{dist}_{L^2(\mu)}(p-v,Q(g)_d)\ge\delta(v)$ for every $d$.
	If $\varepsilon_d<\delta(v)$ then $v$ is infeasible in \eqref{eq:sos}. Since $\varepsilon_d\to0$,
	we obtain $\limsup_{d\to\infty} v_d^\star\le p^\star$.
	
	\emph{Lower limit.}
	Fix $\eta>0$ and set $p_\eta:=p-(p^\star-\eta)\in P(K)$.
	By the $L^2(\mu)$-density of $Q(g)$ in $P(K)$, for large enough $d$, there exists $q\in Q(g)_d$ with
	$\|p_\eta-q\|_{L^2(\mu)}\le\varepsilon_d$.
	Let $v:=p^\star-\eta$ and $r:=p_\eta-q$; then $p-v=r+q$ and $\|r\|_{L^2(\mu)}\le\varepsilon_d$, so $v$ is feasible in \eqref{eq:sos} and $v_d^\star\ge p^\star-\eta$.
	Let $d\to\infty$ and then $\eta\downarrow0$ to get $\liminf_{d\to\infty} v_d^\star\ge p^\star$.
\end{proof}

\begin{assumption}\label{ass:BM}
	Assume that $(K,\mu)$ satisfies a Bernstein-Markov (BM) inequality at degree $2d$, i.e.,
	there exist a constant $c_{2d}\ge1$ such that for all $r\in\R[x]_{2d}$,
	\begin{equation}\label{eq:bm}
		\|r\|_{L^\infty(K)}\ \le\ c_{2d}\,\|r\|_{L^2(\mu)},
	\end{equation}
	see e.g. \cite[Section 4.3.3]{LPP22}.
\end{assumption}

\begin{corollary}\label{cor:BM-lb}
Under Assumptions \ref{ass:refmeas-full} and \ref{ass:BM}, every optimal solution of \eqref{eq:sos} yields a certified lower bound
	\[
	p_d^*(\varepsilon_d) :=\ v_d^\star - c_{2d}\,\varepsilon_d\ \le\ p^\star.
	\]
	Moreover, if $c_{2d}\,\varepsilon_d\to0$, then $p_d^*(\varepsilon_d)\to p^\star$.
\end{corollary}

\begin{proof} 
	For any feasible $(v,r,(s_j))$ in \eqref{eq:sos} and all $x\in K$,
	\[
	p(x)=v+r(x)+\sum_j s_j(x)g_j(x)\ \ge\ v+r(x)\ \ge\ v-\|r\|_{L^\infty(K)}
	\ \ge\ v-c_{2d}\|r\|_{L^2(\mu)}\ \ge\ v-c_{2d}\varepsilon_d.
	\]
	Taking the infimum over $x\in K$ and then the supremum over feasible $v$ gives the bound.
	If $c_{2d}\varepsilon_d\to0$, then $p_d^*\uparrow p^\star$ by Theorem~\ref{thm:L2-conv}.
\end{proof}

\begin{remark} 
	Let $\{b_\alpha\}_{|\alpha|\le 2d}$ be an $L^2(\mu)$ orthonormal polynomial basis and
	define the degree $2d$ Christoffel-Darboux polynomial
	\[ 
		p^{\mu}_{2d}(x):=\sum_{|\alpha|\le 2d}b_\alpha(x)^2.
	\]
	Then the tighest BM constant in \eqref{eq:bm} is $$c_{2d}=\sup_{x\in K}\sqrt{p^{\mu}_{2d}(x)},$$
	see e.g. \cite[Section 4.3.3]{LPP22}.
\end{remark}

\subsection{Lower bounds with penalized formulation}
\label{subsec:l2-penalized-bm-envelope}

Let us now describe an equivalent reformulation of regularized problems \eqref{eq:sos} and \eqref{eq:mom} that does not involve the parameter $\epsilon_d$. Consider the $\ell_2$ penalized formulation:
\begin{equation}\label{eq:pensos}
	\begin{aligned}
			\hat v_d^* \;:=\;
	\sup\ & v - c_{2d}\,\|r\|_{L^2(\mu)}  \\[2pt]
	\text{\rm s.t.}\quad &
	p-v = r  + \sum_{j=0}^m s_j g_j
\end{aligned}
\end{equation}
where the maximization is w.r.t. scalar $v$,
residual polynomial $r \in \R[x]_{2d}$ and SOS multipliers $s_j\in\Sigma[x]_{2(d-d_j)}$. 
As in the proof of Corollary \ref{cor:BM-lb}, any feasible tuple certifies on $K$ the pointwise lower bound
\(
p(x)\ge v-c_{2d}\|r\|_2
\)
for $x \in K$, so $\hat v^*_d\le\min_{x\in K}p(x)$ is a valid lower bound.
SOS problem \eqref{eq:pensos} is the dual to the primal moment problem
\begin{equation}\label{eq:penmom}
	\begin{array}{lll}
		 \hat u^*_d := & \inf &   \ell_y(p)\\[2pt]
		& \text{\rm s.t.}  & \ell_y(1)=1,\\
		&&   M_{d-d_j}(g_j\,y)\succeq0,\quad j=0,1,\dots,m \\
		&& \|y\|_2 \leq c_{2d}. 
	\end{array}
\end{equation}
As above, strong duality ($\hat u^*_d = \hat v^*_d$) follows from Slater conditions.
The next result shows that no tuning of $\varepsilon$ is needed: any optimal solution $v^*, r^*$ of \eqref{eq:pensos} both selects the optimal $\varepsilon_d=c_{2d}\|r^*\|_2$ and returns the tightest BM certified bound $p_d^*=v^*-c_{2d}\|r^*\|_2$.

\begin{lemma}\label{lem:penbound}
It holds
	\begin{equation}\label{eq:bm-envelope-general}
		\hat v^*_d\;=\;\sup_{\varepsilon\ge0}\ (v_d^*(\varepsilon)\;-\;c_{2d}\,\varepsilon).
	\end{equation}
	\end{lemma}
	\begin{proof}
		Let $\mathcal{F}_d$ be the degree-$2d$ SOS–feasible set of triples $(v,r,(s_j))$ with
		$p-v=r+\sum_{j=0}^m s_j g_j$, $r\in\R[x]_{2d}$, $s_j\in\Sigma[x]_{2(d-d_j)}$.
		By definition,
		\[
		\hat v_d^*=\sup_{(v,r,(s_j))\in\mathcal{F}_d}\Big\{\,v-c_{2d}\,\|r\|_{L^2(\mu)}\,\Big\}.
		\]
		For each fixed $r$,
		\[
		v-c_{2d}\,\|r\|_{L^2(\mu)}
		=\sup_{\varepsilon\ge0}\ \Big\{\,v-c_{2d}\,\varepsilon\ :\ \|r\|_{L^2(\mu)}\le \varepsilon\,\Big\},
		\]
		with the supremum attained at $\varepsilon=\|r\|_{L^2(\mu)}$.
		Hence
		\[
		\hat v_d^*=\sup_{(v,r,(s_j))\in\mathcal{F}_d}\ \sup_{\varepsilon\ge0}\ \Big\{\,v-c_{2d}\,\varepsilon\ :\ \|r\|_{L^2(\mu)}\le \varepsilon\,\Big\}.
		\]
		Interchanging the two suprema (over a product set) yields
		\[
		\hat v_d^*=\sup_{\varepsilon\ge0}\ \sup\Big\{\,v-c_{2d}\,\varepsilon\ :\ (v,r,(s_j))\in\mathcal{F}_d,\ \|r\|_{L^2(\mu)}\le \varepsilon\,\Big\}.
		\]
		The inner supremum is precisely $v_d^*(\varepsilon)-c_{2d}\,\varepsilon$ by the definition of $v_d^*(\varepsilon)$, whence
		\[
		\hat v_d^*=\sup_{\varepsilon\ge0}\bigl(v_d^*(\varepsilon)-c_{2d}\,\varepsilon\bigr).
		\]
	\end{proof}
	
\begin{theorem}\label{thm:penconv}
	Under Assumptions \ref{ass:refmeas-full} and \ref{ass:BM},
any optimal solution $v^*, r^*$ of \eqref{eq:pensos} generates a sequence $p^*_d := v^*  - c_{2d} \|r^*\|_2$ of lower bounds which is monotone nondecreasing and converging to $p^*$:
\[
p^*_d \leq p^*_{d+1} \leq \cdots \leq \lim_{d\to\infty} p^*_d = p^*.
\]
\end{theorem}	 

\begin{proof}
	Set $p_d^*:=\hat v_d^*$ (the optimal value of \eqref{eq:pensos}); at any optimizer $(v^*,r^*,(s_j^*))$ we have $p_d^*=v^*-c_{2d}\|r^*\|_{L^2(\mu)}$.
	
	\emph{(i) Certified lower bound.}
	By Assumption~\ref{ass:BM}, for any feasible $(v,r,(s_j))$ and all $x\in K$,
	\[
	p(x)=v+r(x)+\sum_j s_j(x)g_j(x)\ \ge\ v+r(x)\ \ge\ v-\|r\|_{L^\infty(K)}\ \ge\ v-c_{2d}\|r\|_{L^2(\mu)}.
	\]
	Taking $\inf_{x\in K}$ and then $\sup$ over feasible tuples gives $p^*\ge \hat v_d^*=p_d^*$.
	Thus each $p_d^*$ is a valid lower bound.
	
	\emph{(ii) Monotonicity.}
	If $d'<d$, then $\hat v_{d}^*\ge \hat v_{d'}^*$. Indeed,
	since $\R[x]_{2d'}\subset \R[x]_{2d}$ and $Q_{d'}(g)\subset Q(g)_d$, the feasible set of \eqref{eq:pensos} at level $d$ contains that at level $d'$, so the supremum cannot decrease.
	Therefore $p_{d}^*=\hat v_d^*$ is nondecreasing in $d$.
	
	\emph{(iii) Convergence.}
	Fix $\eta>0$. Choose a sequence $\{\varepsilon_d\}_{d\ge1}$ with $\varepsilon_d\downarrow 0$ and $c_{2d}\varepsilon_d\downarrow 0$
	(e.g.\ $\varepsilon_d:=\min\{1/d,\ 1/(d\,c_{2d})\}$).
	By Theorem~\ref{thm:L2-conv} under Assumption~\ref{ass:refmeas-full}, there exists $d_0$ such that
	$v_d^*(\varepsilon_d)\ge p^*-\eta/2$ for all $d\ge d_0$.
	Using Lemma~\ref{lem:penbound} and $c_{2d}\varepsilon_d\le \eta/2$ for $d$ large,
	\[
	p_d^*=\hat v_d^*=\sup_{\varepsilon\ge0}\bigl(v_d^*(\varepsilon)-c_{2d}\varepsilon\bigr)
	\ \ge\ v_d^*(\varepsilon_d)-c_{2d}\varepsilon_d
	\ \ge\ p^*-\eta.
	\]
	Combining with (i) gives $p^*\ge p_d^*\ge p^*-\eta$ for all $d\ge d_0$. Since $\eta>0$ was arbitrary and $(p_d^*)_d$ is nondecreasing, we conclude $p_d^*\uparrow p^*$.
\end{proof}

\subsection{Recovering the standard moment-SOS hierarchy}

At level $d$, the standard moment-SOS hierarchy is
\[
\inf\ \{\,\ell_y(p):\ \ell_y(1)=1,\ M_{d-d_j}(g_jy)\succeq0\ \forall j\ \} = 
\sup\ \{\,v:\ p-v \in Q(g)_d\ \}
\]
Setting $\varepsilon_d=0$ in the constrained form \eqref{eq:sos} enforces $\|r\|_{L^2(\mu)}\le0$, hence $r=0$ and $p-v=r+\sum_js_jg_j = \sum_js_jg_j \in Q(g)_d$.
Dually, the objective function in \eqref{eq:mom} becomes
$\ell_y(p)+r \|y\|_2 = \ell_y(p)$.

In the penalized form \((\ref{eq:pensos})\), we introduce a penalty weight $\lambda>0$:
\[
\sup\ \{\,v-\lambda\|r\|_{L^2(\mu)}:\ p-v=r+\textstyle\sum_js_jg_j\ \}
\]
and our choice \(\lambda=c_{2d}\) is the BM certified instance. As $\lambda\uparrow\infty$,
any optimal sequence must satisfy $\|r\|\to0$ (otherwise the objective would go to $-\infty$),
hence $r\to0$. On the moment side, we have
\[
\inf\ \ell_y(p)\quad\text{s.t.}\quad \ell_y(1)=1,\ M_{d-d_j}(g_jy)\succeq0,\ \ \|y\|_2\le\lambda,
\]
and letting $\lambda\uparrow\infty$ removes the norm constraint.

\section{Examples}\label{sec:examples}

In the following examples we used the MOSEK semidefinite solver and we modeled the moment-SOS hierarchy in Python
on a standard PC running under Ubuntu 24.
Numerical values are reported to 5 significant digits.

\subsection{Motzkin}\label{ex:motzkin2}

Let us apply the regularized hierarchy to the Motzkin polynomial
\(p_M(x)=1-3x_1^2x_2^2+x_1^4x_2^2+x_1^2x_2^4\) of Example \ref{ex:motzkin}. Since the problem is unconstrained, we can choose the normalized Gaussian (zero mean unit variance).  
To express the moment matrices and SOS decompositions we use orthonormal Hermite polynomials $b_k$ so that \(\|q\|_{L^2(\mu)}\) is the Euclidean norm of the coefficient vector of \(q\) in the \(\{b_k\}\) basis, recall
\eqref{eq:normcoef}. The Hermite polynomials satisfy $b_0(x)=1$, $b_1(x)=x$ and the three-term recurrence $x b_k(x) = \sqrt{k+1} b_{k+1}(x) + \sqrt{k} b_{k-1}(x)$, $k \geq 1$.

At relaxation order \(d=3\) (degree \(2d=6\)), after solving the moment-SOS problem \eqref{eq:sos}-\eqref{eq:mom} we obtain the following numerical bounds: 
\[
\begin{array}{c|ccccccc}
	\varepsilon_3 & 1 & 10^{-1} & 10^{-2} & 10^{-3} & 10^{-4} & 10^{-5} & 10^{-6} \\
	\hline
	v^*_3 & 1.2472 & -0.68628 & -3.0515 & -10.104 & -32.325 & -102.57 & -324.69
\end{array}
\]
We observe that the bounds follow a square-root law
\(v^*_3(\varepsilon_3)\sim -\kappa_3 \varepsilon^{-1/2}_3\) with $\kappa_3 \approx 0.32452$.
Because \(p_M\notin\Sigma_6\), the best degree-6 SOS approximation sits on a rank-2 face generated by the cubic forms
\(f_{\pm}=x_1^2x_2\pm x_1x_2^2\), with \(s_0=\tfrac12 f_+^2+\tfrac12 f_-^2\).
The feasible SOS variations issuing from \(s_0\) at first order are the tangent directions generated by the two active squares, i.e., linear combinations of \(f_+u\) and \(f_-v\) (with \(u,v\) cubic). Writing \(p_M=s_0+w\), the residual \(w=1-3x_1^2x_2^2\) decomposes into a part that can be created by these first-order SOS variations (tangent image) and a part that  cannot (transverse to that image). The parameter \(\kappa_3\) quantifies, in the Gaussian \(L^2\) metric used by the hierarchy, the size of this transverse component of \(w\) relative to the constant shift direction. Geometrically, it measures how stiff the cone \(\Sigma_6\) is at the face through \(s_0\) against moving along the affine line \(p_M - v\), and it is precisely this second-order stiffness that produces the square-root law. In short, \(\kappa_3\) is the intrinsic curvature-controlled obstruction, at \(s_0\), to absorbing the degree-6 residual of \(p_M\) by first-order SOS deformations.

Raising the relaxation order  to \(d=4\) enriches the active face of the SOS cone by adding a \emph{quartic} square direction. Concretely, besides the rank-2 face through $s_0$
the degree-8 cone \(\Sigma_8\) also contains the square of the quartic
\(g(x) = 1-\tfrac32 x_1^2 x_2^2\), i.e.\ \(g^2\in\Sigma_8\).
This extra square direction lets one cancel, at \(s_0\), not only the tangent (first-order) component of the residual \(w=1-3x_1^2x_2^2\), but also its second-order effect along the enlarged face spanned by the two cubic squares and the quartic square. The first non-removable piece appears at third order in a local SOS expansion. Geometrically, this means the affine line \(p_M-v\) now meets \(\Sigma_8\) with third-order contact at the augmented face, and the remaining transverse obstruction is governed by this third-order curvature. As \(\varepsilon\downarrow0\), this yields the cubic-root law $v^*_4(\varepsilon) \sim -\kappa_4  \varepsilon^{-1/3}$
with a prefactor $\kappa_4 \approx 0.22902$ determined by that third-order stiffness in the Gaussian \(L^2\) metric. Numerically, the bounds behave in line with this picture:
\[
\begin{array}{c|ccccccc}
	\varepsilon_4 & 1 & 10^{-1} & 10^{-2} & 10^{-3} & 10^{-4} & 10^{-5} & 10^{-6} \\
	\hline
	v^*_4 &
	1.2722 & -0.081974 & -0.69080 & -1.9431 & -4.6598 & -10.518 & -23.142
\end{array}
\]

At \(d=5\), the additional degree-5 square directions do not remove the remaining {third-order} transverse component at that face: the affine line \(p_M-v\) continues to have third-order contact with the (now larger) cone, so the same cubic-root law persists: $v_5(\varepsilon)\ \sim\ -\,\kappa_5\,\varepsilon^{-1/3}$  but
with a smaller prefactor $\kappa_5 \approx 0.053771$ reflecting the reduced stiffness of \(\Sigma_{10}\) near the face containing \(s_0\):
\[
\begin{array}{c|ccccccc}
	\varepsilon_5 & 1 & 10^{-1} & 10^{-2} & 10^{-3} & 10^{-4} & 10^{-5} & 10^{-6} \\
	\hline
	v^*_5 &
	\phantom{-}1.4503 & \phantom{-}0.18558 & -0.073781 & -0.45841 & -1.2176 & -2.6123 & -5.1260
\end{array}
\]
At higher relaxation orders $d$ the trend is similar with decreasing prefactors $\kappa_d$.

\subsection{Origin}

Consider Example \ref{ex:origin} with $p(x)=x$ and $g(x)=-x^2$, so that $K=\{0\}$ reduces to the origin, with the obvious solution $x^*=p^*=0$. Let $\mu$ be the normalized Gaussian as in Section \ref{ex:motzkin2}.

Fix a relaxation order \(d\in\mathbb{N}\) and a regularization parameter \(\varepsilon_d>0\).
The (dual) SOS problem is
\[
\max_{v,\;r,\;s_0,\;s_1}\ \{\,v:\ x-v=r+s_0-s_1\,x^2,\ \ s_0\in\Sigma[x]_{2d},\ s_1\in\Sigma[x]_{2(d-1)},\ \ \|r\|_{L^2(\mu)}\le\varepsilon_d\},
\]
and the (primal) pseudo-moment problem reads
\begin{equation}\label{eq:dual-point}
	\min_{y}\ \ell_y(p)+\varepsilon_d\|y\|_2
	\quad\text{s.t.}\quad
	y_0=1,\ \ M_d(y)\succeq0,\ \ M_{d-1}(-x^2\,y)\succeq0,
\end{equation}
where \(y=(y_0,\dots,y_{2d})\) acts by \(\ell_y(q)=\sum_{k=0}^{2d} q_k\,y_k\) on Hermite coefficients \(q_k\), and
\([M_r(y)]_{ij}=\ell_y(b_i b_j)\), \([M_{r-1}(-x^2\,y)]_{ij}=\ell_y(-x^2 b_i b_j)\).
Let $v^*_d$ denote the value of the primal-dual problems.

Now observe that if \(y\) satisfies the constraints in \eqref{eq:dual-point}, then
$\ell_y(x^2q^2)=0$, $\forall q\in\mathbb{R}[x]_{d-1}.$
In particular, in the Hermite basis one has $y_k \;=\; b_k(0)$, $k=0,1,\dots,2d$. That is, \(y\) coincides with the truncated moment sequence of the Dirac mass \(\delta_0\). One has \(p(x)=x=b_1(x)\), hence \(\ell_y(p)=y_1=b_1(0)=0\).
For the Hermite polynomials,
\(
b_{2m+1}(0)=0
\) and
\(
b_{2m}(0)=(-1)^m\sqrt{(2m)!}/(2^m m!)
\).
It follows that the optimal value is
\[
v_d^\star\;=\;\varepsilon_d\,\big\|y\big\|_2
\;=\;
\varepsilon_d\,\sqrt{\sum_{k=0}^{2d} b_k(0)^2}
\;=\; \varepsilon_d\,\sqrt{\displaystyle\sum_{m=0}^{d}\frac{\binom{2m}{m}}{4^m}}.
\]
Solving the corresponding regularized moment-SOS hierarchy gives the following values matching with the analytic values:
\[
\begin{array}{c|ccccccccc}
	d & 2 & 3 & 4 & 5 & 6 & 7 & 8 & 9 & 10\\ \hline
	v_d^*/\varepsilon_d &
	1.3693 & 1.4790 & 1.5687 & 1.6453 & 1.7125 & 1.7726 & 1.8271 & 1.8772 & 1.9236
\end{array}
\]
As \(d\to\infty\), since \(\binom{2m}{m}/4^m\sim (\pi m)^{-1/2}\), $v_d^* \sim \frac{\sqrt{2}}{\pi^{1/4}}\,\varepsilon_d\,d^{1/4}$
and hence if $\varepsilon_d=o(d^{-1/4})$ then  $v_d^\star\to p^*=0$. Note that if the decrease of $\varepsilon_d$ is too slow, e.g. $\varepsilon_d=d^{-1/4}$, then $v^*_d$ does not converge to $p^*$.

In the present case $K=\{0\}$, even if $\mu$ is Gaussian (so $w>0$ everywhere), every relatively open subset $V\subset K$
	is of the form $V=\{0\}$ and thus has Lebesgue measure $|V|=0$ and $\mu(V)=0$.
	Consequently, the key estimate used in the proof of Theorem~\ref{thm:L2-conv},
	$\|r-q\|_{L^2(\mu)}^2 \;\ge\; \int_V |r-q|^2\,w\,dx \;\ge\; \gamma^2\,\mu(V) \;>\;0$
	fails because $\mu(V)=0$ for every nonempty relatively open $V\subset K$.
Thus the upper bound argument (which needs a set of  {positive} $\mu$-measure where $p-v$ is strictly negative)
cannot be carried out on $K=\{0\}$.

On the other hand, the sharp BM constant \eqref{eq:bm} at degree \(2d\) is
\[
c_{2d}\;=\;\sqrt{p^{\mu}_{2d}(0)}
\;=\;\sqrt{\sum_{k=0}^{2d} b_k(0)^2}
\;=\;\sqrt{\sum_{m=0}^{d}\frac{\binom{2m}{m}}{4^m}}.
\]
Therefore the BM certified bound of Corollary~\ref{cor:BM-lb} becomes
\[
p_d^*\ :=\ v_d^* - c_{2d}\,\varepsilon_d
\;=\;
\varepsilon_d\sqrt{\sum_{m=0}^{d}\frac{\binom{2m}{m}}{4^m}}
\ -\
\varepsilon_d\sqrt{\sum_{m=0}^{d}\frac{\binom{2m}{m}}{4^m}}
\;=\;0.
\]
Since the true optimum is \(p^*=0\), we have
$p_d^*=p^*$ for all $d\in\mathbb{N}$, $\varepsilon_d>0$. For this degenerate POP, the BM certificate is exact (tight) at every relaxation order and every regularization level.

\subsection{Stengle}

\begin{figure}[t]
	\centering
	\includegraphics[width=.78\linewidth]{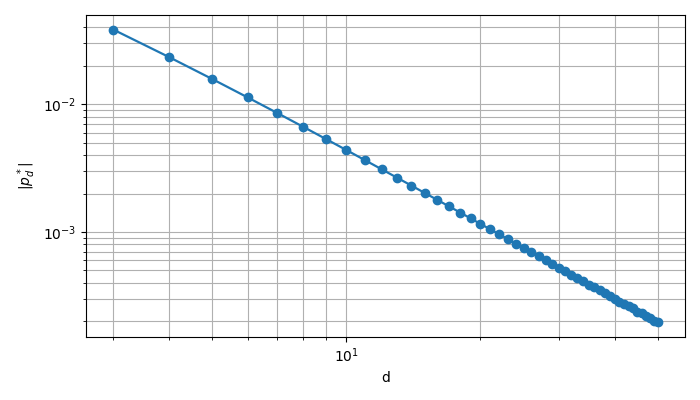}
	\caption{Log-log plot of the gaps \(|p_d^*-p^*|\) versus the relaxation order \(d\)
		for the Stengle POP.
		The near-linearity indicates an algebraic decay \(|p_d^*-p^*|\approx C\,d^{-2}\).
		\label{fig:regmomsos-stengle}}
\end{figure}

As in Example \ref{ex:stengle}, let $p(x)=1-x^2$, and $g(x)=(1-x^2)^3$ so that $K=[-1,1]$. If $\mu$ is the arcsine measure, then its orthonormal basis are Chebyshev polynomials $b_0=T_0$, $b_k=\sqrt{2}\,T_k$ for $k\ge1$.
The corresponding BM constant is
\[
c_{2d}\;=\;\sup_{x\in K}\sqrt{\sum_{k=0}^{2d} b_k(x)^2} \;=\;\sqrt{4d+1}.
\]
Solving the penalized mom-SOS hierarchy \eqref{eq:pensos}-\eqref{eq:penmom}, we obtain the lower bounds $p^*_3 \approx -3.8252 \cdot 10^{-2} \leq p^*_{50} \approx -1.9605\cdot 10^{-4} \leq p^*=\min_{x\in[-1,1]}(1-x^2)=0$. Figure \ref{fig:regmomsos-stengle} shows neatly the 
algebraic decay of the gaps \(|p_d^*-p^*|\approx C\,d^{-2}\).

\subsection{Prestel and Delzell} 

Consider the POP of Example \ref{ex:prestel} with
$p=\tfrac{17}{4}-x^2_1-x^2_2$, $g_1(x)=x_1-\tfrac12$,
$g_2(x)=x_2-\tfrac12$,
$g_3(x)=1-x_1x_2$
and where $Q(g)$ is a non-Archimedean quadratic module. The minimum is $p^*=0$ attained at $x \in \{(1/2,2), (2,1/2)\}$.

We use the affine map \(x_i \mapsto \tfrac34\,x_i+\tfrac54\)  so that
\([-\!1,1]^2\) is sent to \([\tfrac12,2]^2\supset K\).
On \([-1,1]\) we use the arcsine reference measure
\[
d\mu(y)=\frac{1}{\pi}\frac{dy}{\sqrt{1-y^2}},
\]
for which the Chebyshev polynomials of the first kind are orthogonal. Using the  orthonormalized 1D basis
\(b_0=T_0\), \(b_k=\sqrt2\,T_k\) for \(k\ge1\), in 2D we use the tensor basis
\(b_{i,j}=b_ib_j\). Since $\sup_{x \in [-1,1]} |b_0(x)|=1$ and $\sup_{x \in [-1,1]} |b_k(x)|=\sqrt{2}$ for $k\geq 1$, for an expansion
\[
r(x) = \sum_{i+j \leq 2d} r_{ij} b_{i,j}(x)
\]
the Cauchy-Scharwz inequality gives
\[
\sup_{x \in [-1,1]^2} |r(x)| \leq \sqrt{\sum_{i+j \leq 2d} \alpha_i \alpha_j} \sqrt{\sum_{i+j \leq 2d} r^2_{ij}}
\]
where $\alpha_0=1$ and $\alpha_k=2$ for $k\ge 1$. Therefore the BM constant is
$c_{2d}^2=\sum_{i+j\le 2d}\alpha_i\alpha_j.$
Counting pairs: exactly one of \(i,j\) positive gives \(4d\) pairs with weight \(2\);
both positive gives \((2d+1)(d+1)-(4d+1)=2d^2-d\) pairs with weight \(4\); and \((0,0)\) contributes \(1\).
Hence $c_{2d}^2
=1+8d+(8d^2-4d)
=8d^2+4d+1$, i.e.
\[
c_{2d}=\sqrt{8d^2+4d+1}.
\]
The table below reports the bounds and internal quantities
for \(d=1,2,3\) obtained by solving the penalized moment-SOS hierarchy \eqref{eq:pensos}-\eqref{eq:penmom}:
\[
	\begin{tabular}{c|r r r}
		$d$ & \(\;p_d^*=v^*-c_{2d}t^*\) & \(\;v^*\) & \(\;t^*\)\\
		\hline
		1 & \(-0.40399\)           & \(1.4382\)              & \(0.51093\) \\
		2 & \(-2.1928\times10^{-3}\) & \(5.9888\times10^{-3}\) & \(1.2777\times10^{-3}\) \\
		3 & \(-1.7033\times10^{-9}\) & \(5.1568\times10^{-8}\) & \(5.7781\times10^{-9}\) \\
		\end{tabular}
\]
The certified lower bounds are negative and monotonically increasing to \(p^*=0\),  
demonstrating rapid convergence of the regularized hierarchy:
from \(-4.0\cdot10^{-1}\) at \(d=1\) to about \(-2.2\cdot10^{-3}\) at \(d=2\),
and essentially machine precision at \(d=3\).
Concurrently, the optimizer chooses a small \(v^*\) and drives the residual norm \(t^*\) to zero,
so that \(v^*-c_{2d}t^*\uparrow0\).
The Chebyshev-arcsine tensor basis together with exact product rules is numerically  stable,
and the explicit BM constant \(c_{2d}=\sqrt{8d^2+4d+1}\) provides tight control on the sup-norm via the
\(\ell_2\)-penalty on coefficients.

\section*{Acknowledgment}

This paper benefited from feedback by J. B. Lasserre, who pointed out the reference \cite{DX14} for the relationship between
the Carleman condition of the $L^2$ density of polynomials.
The author also acknowledges the help of AI in the technical developments.


\begin{thebibliography}{XX}
\bibitem{BS24}
L. Baldi, L. Slot.
Degree bounds for Putinar's positivstellensatz on the hypercube. SIAM J. Appl. Alg. Geom. 8(1):1-25, 2024.
\bibitem{BC11}
H. H. Bauschke, P. L. Combettes. Convex analysis and monotone operator theory in Hilbert spaces. Springer, 2011.
\bibitem{DX14}
C. Dunkl, Y. Xu. Orthogonal polynomials of several variables, 2nd edition. Cambridge Univ. Press, 2014.
\bibitem{HKL20}
D. Henrion, M. Korda and J. B. Lasserre.
The moment-SOS hierarchy.
World Scientific, 2020.
\bibitem{L01}
J. B. Lasserre. Global optimization with polynomials and the problem of moments.
SIAM J. Optim. 11(3):796--817, 2001.
\bibitem{L06}
J. B. Lasserre. A sum of squares approximation of nonnegative polynomials. SIAM J. Optim 16:751–765, 2006.
\bibitem{L10}
J. B. Lasserre. Moments, positive polynomials and their applications. Imperial
College Press, 2010.
\bibitem{L13}
J. B. Lasserre.
The K-moment problem for continuous linear functionals. Trans. Amer. Math. Soc. 365(5):2489-2504, 2013.
\bibitem{LN07}
J. B. Lasserre, T. Netzer.
SOS approximations of nonnegative polynomials via simple high degree perturbations. Math. Zeitschrift 256:99-112, 2007.
\bibitem{LPP22}
J. B. Lasserre, E. Pauwels, M. Putinar.
The Christoffel-Darboux kernel for data analysis. Cambridge Univ. Press, 2022.
\bibitem{M08}
M. Marshall.
Positive polynomials and sums of squares.
AMS, 2008.
\bibitem{N08}
T. Netzer. Positive polynomials, sums of squares
and the moment problem. PhD Thesis, Univ. Konstanz, 2008.
\bibitem{N23}
J. Nie. Moments and polynomial optimization. SIAM, 2023.	
\bibitem{T24}
T. Theobald. Real algebraic geometry and optimization. AMS, 2024.
\bibitem{PD01}
A. Prestel, C. N. Delzell. Positive polynomials. Springer, 2001.
\bibitem{P93}
M. Putinar. Positive polynomials on compact semi-algebraic sets. Indiana Univ.
Math. J. 42:969-984, 1993.
\bibitem{S17}
K. Schm\"udgen. The moment problem. Springer, 2017.
\bibitem{S96}
G. Stengle. Complexity estimates for the Schm\"udgen Positivstellensatz, J. Complex. 12:167-174, 1996.
\end{thebibliography}
\end{document}